\documentclass[reqno, psamsfonts, 11pt]{amsart}
\usepackage{mathabx}    
\usepackage{mathrsfs}  
\usepackage{amsmath,amssymb,amsthm}
\usepackage[abbrev,msc-links]{amsrefs}   
\usepackage{tikz}
\usepackage{amsmath}
\usepackage{enumerate}
\usepackage{color}
\usepackage{amssymb,amsmath}
\usepackage{lineno}
\usepackage{hyperref}
\usepackage{amsthm}
\usepackage{enumitem}

\usepackage{wasysym}
\usepackage{graphicx}
\usepackage{arydshln}
 \usepackage{tikz}
 \usepackage{float}
 \usepackage{mathtools}

\newtheorem{theorem}{Theorem}[section]
\newtheorem{lemma}[theorem]{Lemma}

\newcommand{\mbb}{\mathbb}
\newcommand{\Pro}{\mbb{P}}
\newcommand{\Ex}{\mbb{E}}
\renewcommand{\P}{\mathcal{P}}

\usepackage{tikz}
\tikzstyle{aNode} = [circle, fill = black]
\tikzstyle{bNode} = [circle,draw = black, thick]

\newcommand{\pcherry}[1]{%
\begin{tikzpicture}[inner sep = 1pt, #1]%
\node (1) at (0,-2) [aNode]{};
\node (3) at (1.5,-2) [aNode]{};
\node (2) at (0.75,-1) [aNode]{};
\draw  (1) -- (2);
\draw  (2) -- (3);
\end{tikzpicture}%
}

\newcommand{\ppoints}[1]{%
\begin{tikzpicture}[inner sep = 1pt, #1]%
\node (1) at (0,-2) [aNode]{};
\node (3) at (1.5,-2) [aNode]{};
\node (2) at (0.75,-1) [aNode]{};
\end{tikzpicture}%
}

\newcommand{\pedge}[1]{%
\begin{tikzpicture}[inner sep = 1pt, #1]%
\node (1) at (0,-2) [aNode]{};
\node (3) at (1.5,-2) [aNode]{};
\node (2) at (0.75,-1) [aNode]{};
\draw  (1) -- (3);
\end{tikzpicture}%
}

\def\cherry{\pcherry{scale=0.13}}

\def\points{\ppoints{scale=0.13}}
\def\edge{\pedge{scale=0.14}}

\begin{document}


\title{Hamiltonicity in Cherry-quasirandom 3-graphs} 
\thanks{JH is partially supported by Simons Collaboration Grant for Mathematicians \#630884.}
\author{Luyining Gan}
\author{Jie Han}
\address{Department of Mathematics and Statistics, Auburn University, Auburn, AL, 36849, USA. Email: \tt{lzg0027@auburn.edu}}
\address{Department of Mathematics, University of Rhode Island, Kingston, RI, 02881, USA. Email: \tt{jie\_han@uri.edu}}
\date{\today}

%
%

\begin{abstract}

We show that for any fixed $\alpha>0$, cherry-quasirandom 3-graphs of positive density and sufficiently large order $n$ with minimum vertex degree $\alpha \binom n2$ have a tight Hamilton cycle.
This solves a conjecture of Aigner-Horev and Levy.
\end{abstract}
\maketitle

\section{Introduction}
The study of Hamilton cycles is a central topic in graph theory with a long and profound history. 
In recent years, researchers have worked on extending the classical theorem of Dirac on Hamilton cycles to hypergraphs 
and we refer to~\cite{BHS, GPW, HZ1, HZ2, RR, BMSSS1, BMSSS2, RRRSS, HZ_forbidHC, HHZ_cycle} for some recent results and to \cite{KuOs14ICM, RR, zsurvey} for excellent surveys  on this topic.

In this paper we restrict ourselves to 3-uniform hypergraphs (3-graphs), where each (hyper)edge contains exactly three vertices.
For a 3-graph $H$ and a vertex set $S\subseteq V(H)$, $\deg_H(S)$ is defined to be the number of edges containing $S$.
The \emph{minimum codegree $\delta _{2} (H)$} of $H$ is the minimum of $\deg_{H} (S)$ over all pairs $S$ of vertices in $H$, and the \emph{minimum degree $\delta_1(H)$} of $H$ is the minimum of $\deg_{H} (v)$ over all vertices $v\in V(H)$.
A $3$-graph $C$ is called a \emph{tight cycle} if its vertices can be ordered cyclically such that every 3 consecutive vertices in this ordering define an edge of $C$, which implies that every two consecutive edges intersect in two vertices.
We say that a $3$-graph contains a \emph{tight Hamilton cycle} if it contains a tight cycle as a spanning subgraph. 
A tight path $P$ has a sequential order of vertices $v_1v_2\dots v_{p-1}v_p$ such that every 3 consecutive vertices form an edge, where the ends of $P$ are ordered pairs $(v_2, v_1)$ and $(v_{p-1}, v_p)$.

The study of quasirandom graphs has been a fruitful area since introduced in \cite{Chung1989,Thomason1987a,Thomason1987b}, and we recommend the readers to the excellent survey~\cite{Krivelevich2006}.
However, the canonical definitions for quasirandom hypergraphs (extending~\cite{Chung1989}) have been completely settled only recently~\cite{Horev2018, Towsner2017}.
In this note we focus on the so-called `cherry-quasirandom 3-graphs' defined as follows.
An $n$-vertex 3-graph $H$ is called $(\rho,d)_{\cherry}$-dense if
\[
e_H(\vec{G}_1, \vec{G}_2):= |\{(x,y,z)\in \P_2(\vec{G}_1, \vec{G}_2): \{x, y, z\}\in E(H)\}| \ge d|\P_2(\vec{G}_1, \vec{G}_2)| - \rho n^3
\]
for every $\vec G_1,\vec G_2 \subseteq V(H) \times V(H)$, where
\[
\P_2(\vec{G}_1, \vec{G}_2) : = \{(x,y,z)\in V(H)^3: (x,y)\in \vec G_1, (y,z)\in \vec G_2\}.
\]

Aigner-Horev and Levy proved the following result on tight Hamiltonicity in $(\rho,d)_{\cherry}$-dense 3-graphs.
\begin{theorem}\cite{Horev}\label{thm:cod}
For every $d,\alpha \in (0,1]$, there exist an integer $n_0$ and a real $\rho >0$ such that the following holds for all $n \geq n_0$. 
Let $H$ be an $n$-vertex $(\rho,d)_{\cherry}$-dense $3$-graph satisfying $\delta_2(H) \geq \alpha n$. Then, $H$ has a tight Hamilton cycle. 
\end{theorem}

They also showed that for $\alpha+d>1$ the $(\rho,d)_{\cherry}$-denseness together with $\delta_1(H) \geq \alpha \binom n2$ implies tight Hamiltonicity and asked~\cite[Conjecture 1.6]{Horev} if the condition $\alpha+d>1$ can be dropped.
In this note we verify this conjecture.

\begin{theorem}\label{thm:main}
For every $\alpha,d \in (0,1]$ there exist an $n_0$ and $\rho>0$ such that the following holds for all $n \geq n_0$. 
Let $H$ be an $n$-vertex $(\rho,d)_{\cherry}$-dense $3$-graph satisfying $\delta_1(H) \geq \alpha \binom n2$. 
Then, $H$ has a tight Hamilton cycle. 
\end{theorem}

There are weaker versions of quasirandomness for 3-graphs compared with $\cherry$-denseness, namely, $\edge$-denseness and $\points$-denseness.
An $n$-vertex 3-graph $H$ is called $(\rho,d)_{\points}$-dense if
\[
e_H(X, Y, Z):= |\{(x,y,z)\in X \times Y\times Z : \{x, y, z\}\in E(H)\}| \ge d|X| |Y| |Z| - \rho n^3
\]
for every $X, Y, Z\subseteq V(H)$;
an $n$-vertex 3-graph $H$ is called $(\rho,d)_{\edge}$-dense if
\[
e_H(X, G):= |\{(x,(y,z))\in X \times G : \{x, y, z\}\in E(H)\}| \ge d|X| |G| - \rho n^3
\]
for every $X\subseteq V(H)$ and $G\subseteq V(H) \times V(H)$.
It is known that the $\cherry$-denseness in Theorems~\ref{thm:cod} and~\ref{thm:main} cannot be replaced by either of these two weaker ones -- namely, degenerate choices of $\alpha$ and $d$ do not guarantee tight Hamiltonicity under these two notions of quasirandomness.
In this sense, the $\cherry$-denseness in these two theorems is best possible.
In contrast, for a weaker notion of Hamiltonicity, namely, the \emph{loose cycles}, Lenz, Mubayi and Mycroft~\cite{Lenz2016} proved that degenerate choices of $\alpha$ and $d$ already force loose Hamiltonicity under $\points$-denseness.
Very recently, Ara\'ujo, Piga and Schacht~\cite{Araujo2019} annouced that for any $\alpha>0$ and $d> 1/4$, having minimum vertex degree $\alpha \binom n2$ and being $(\rho,d)_{\edge}$-dense guarantee tight Hamiltonicity.

\subsection{Proof ideas}
Let us briefly talk about our proof here.
Following other recent work on Hamilton cycles, we use the absorption method, which roughly splits the proof into the following three steps.
Let $H$ be an $n$-vertex $(\rho,d)_{\cherry}$-dense $3$-graph satisfying $\delta_1(H) \geq \alpha \binom n2$. 

\begin{itemize}
\item Absorber lemma: every vertex $v$ in $H$ has many \emph{absorbers}, namely, a constant-length tight path that can include $v$ as an interior vertex or leave $v$ out;
\item Connection lemma: every two ordered pairs of vertices can be connected by a constant-length tight path;
\item Path cover lemma: almost all vertices of the 3-graph can be covered by a constant number of vertex-disjoint tight paths.
\end{itemize}

It is straightforward to prove the path cover lemma for quasirandom 3-graphs.
The proof of Theorem~\ref{thm:cod} relies on the fact that \emph{all} pairs of vertices have a good codegree (namely, $\alpha n$), which, together with the cherry-denseness, allows them to employ the \emph{cascade} method to establish a connection lemma.
Our main advance is to observe that as $H$ is $\cherry$-dense, \emph{almost all} pairs of vertices of $H$ have a good codegree.
Moreover, a `shaving' technique (e.g., \cite[Lemma 8.8]{Han2017}) shows that we can find a spanning subgraph of $H$ where almost all pairs have codegree $dn/3$ and other pairs have degree 0.
These allow us to stick to these high codegree pairs and employ the cascades for connections, and actually such a result has been proven in~\cite[Lemma 3.23]{Horev}.

\section{Tools}
In this section we prove the lemmas needed for the proof of Theorem~\ref{thm:main}.
Note that by definition, if $H$ is an $n$-vertex $(\rho,d)_{\cherry}$-dense $3$-graph, then any induced subgraph of $H$ on $\alpha n$ vertices is $(\rho/\alpha^3,d)_{\cherry}$-dense and we will use this simple fact without further references.

Throughout the rest of the paper, we will refer to tight paths as just paths.
Given a 3-graph $H$, let $\partial H := \{(x,y): \deg_H(xy) >0\}$ be the \emph{shadow} of $H$.
%
For any $v\in A$, define $\deg_H(v, A): = \big| N_H(v) \cap \binom A2 \big|$ and for each pair of vertices $x, y\in A$, let
$\deg_{H}(xy, A) : = | N_{H}(x, y) \cap A |$.

We use the following result proved in~\cite[Lemma 8.8]{Han2017}.
Note that its original version does not include an estimate on the loss of the number of edges which actually follows from its proof. 

\begin{lemma} \cite{Han2017} \label{lem:strongdense}
Let $n\ge 6$ and $0<\mu, \theta <1$.
Let $H$ be an $n$-vertex $3$-graph with $\deg_{H}(S)\ge \mu(n-2)$ for all but at most $\theta \binom{n}{2}$ pairs $S$. 
Then $H$ contains a spanning subgraph $H'$ with $e(H \setminus H') \le 48\theta^{1/4} \binom n3$ and either $\deg_{H'}(S)\ge (\mu-8\theta^{1/4})(n-2)$ or $\deg_{H'}(S) = 0$.
Moreover, $|\partial H'|\ge (1-\theta - \theta^{1/4})\binom n2$, namely, the number of $S$ with $\deg_{H'}(S)= 0$ is at most $(\theta+\theta^{1/4})\binom n2$.
\end{lemma}

The following lemma defines a spanning subgraph of an $n$-vertex $3$-graph $H$, which will be crucial in our proof.
We note that in this lemma $\edge$-denseness is sufficient.

\begin{lemma}\label{lem:shave}
Let $d>0$ and $\rho \le d^5/(3^{40} 2^{5})$.
Let $H$ be an $n$-vertex $(\rho,d)_{\cherry}$-dense $3$-graph.
Then there exists a spanning subgraph $H'$ of $H$, satisfying the following properties.
\begin{enumerate}[label=(\arabic*)]
\item For any pair $S$ of vertices, either $\deg_{H'}(S)\ge d n/3$ or $\deg_{H'}(S)= 0$. 
Moreover, $|\partial H'|\ge (1-\rho^{1/5})\binom n2$, namely, the number of $S$ with $\deg_{H'}(S)= 0$ is at most $\rho^{1/5} \binom{n}{2}$. \label{item:1}
\item $H'$ is $(\rho^{1/5},d)_{\cherry}$-dense. \label{item:2}
\end{enumerate}
\end{lemma}

\begin{proof}
Let $\mathcal{S}$ be the collection of pairs with $\deg_{H}(S) < d (n-2)/2$.
Let $\vec{\mathcal{S}} := \{(x, y): xy\in \mathcal{S}\}$. 
So $|\vec{\mathcal S}|=2|\mathcal S|$.
Since $H$ is $(\rho,d)_{\cherry}$-dense, we have
\[
\frac{d(n-2)}{2} \cdot 2 |{\mathcal S}| \ge e_H({\vec{\mathcal S}}, V(H)^2)   \ge d |\P_2( {{\vec{\mathcal S}}, V(H)^2)| - \rho n^3  = 2 d |\mathcal S}| n - \rho n^3.
\]
The above inequalities imply that
\[
 |{\mathcal S}| \le \frac{\rho n^3}{d(n+2)} \le \frac{\rho n(n-1)}{d}=\frac{2\rho}{d}\binom{n}{2}.
 \]
 
%
Let $H'$ be the spanning subgraph of $H$ returned by Lemma \ref{lem:strongdense} with $\mu = d/2$ and $\theta = 2\rho/d$. 
So we have $e(H \setminus H') \le 48({2\rho}/{d})^{1/4} \binom n3$, $|\partial H'|\ge (1-2\rho/d - ({2\rho}/{d})^{1/4})\binom n2\ge (1-\rho^{1/5})\binom n2$ and 
\[
\deg_{H'}(S)\ge (d/2 - 8( {2\rho}/{d})^{1/4})(n-2) \ge (d/2-d/7)(n-2) \ge dn/3.
\]
Thus, \ref{item:1} holds.

Since $H$ is $(\rho,d)_{\cherry}$-dense, for every $\vec G_1,\vec G_2 \subseteq V(H) \times V(H)$,
we have
\begin{align*}
e_{H'}(\vec G_1,\vec G_2) & \ge e_{H}(\vec G_1,\vec G_2) - e(H \setminus H')\ge d |\P_2(\vec G_1, \vec G_2)| - \rho n^3 -48({2\rho}/{d})^{1/4} \binom n3 \\
&\ge d |\P_2(\vec G_1, \vec G_2)| - (\rho + 8({2\rho}/{d})^{1/4})n^3 \ge d |\P_2(\vec G_1, \vec G_2)| - \rho^{1/5} n^3.
\end{align*}
Thus $H'$ is $(\rho^{1/5},d)_{\cherry}$-dense.
\end{proof}

Let $H$ be a $3$-graph. 
For $v \in V(H)$, a quadruple $(x,y,z,w) \in V(H)^4$ is said to be a $v$-{\em absorber} if 
$\{x,y,z\},\{y,z,w\}, \{v,x,y\},\{v,y,z\},\{v,z,w\} \in E(H)$.
We state and use \cite[Lemma 4.2]{Horev} (in a weaker form) and refine the absorbers it gives in the next lemma.

\begin{lemma} \cite[Lemma 4.2]{Horev} \label{lem:Horev}
For every $\alpha, d \in (0,1]$, there exist $\rho >0$ and $c>0$ such that the following holds for any sufficiently large integer $n$.
Let $H$ be an $n$-vertex $(\rho,d)_{\cherry}$-dense $3$-graph satisfying $\delta_1(H) \geq \alpha \binom{n}{2}$, and let $v \in V(H)$. Then, there are at least 
$
c n^4
$
$v$-absorbers in $H$.
\end{lemma}

\begin{lemma} \label{lem:absorber}
For every $\alpha, d\in (0,1]$, there exists $\rho>0$ such that the following holds for any sufficiently large integer $n$.
Let $H$ be an $n$-vertex $(\rho,d)_{\cherry}$-dense $3$-graph satisfying $\delta_1(H)\ge \alpha \binom{n}{2}$, and let $H'$ be a spanning subgraph of $H$ satisfying Lemma \ref{lem:shave} \ref{item:1}. 
Let $v \in V(H)$.
Then, for any $W \subseteq V(H)$ with $|W| \le \alpha n /4$, there exists a $v$-absorber $v_1v_2v_3v_4$ in $V(H)\setminus W$ such that $\deg_{H'}(v_1v_2) \ge d n/3$ and $\deg_{H'}(v_3v_4) \ge d n/3$.
\end{lemma}

\begin{proof}
Apply Lemma \ref{lem:Horev} with $\alpha/2$ and $d$, and obtain $\rho'>0$ and $c>0$.
Let $\rho = \min\{ c^5(1-\alpha/4)^{20}, \rho'(1-\alpha/4)^3, d^5/(3^{40} 2^{5})\}$, $H_1 := H[(V(H)\setminus W)\cup \{v\}]$ and denote its order by $n_1 (\ge n - \alpha n/4)$.
Then $H_1$ is $(\rho',d)_{\cherry}$-dense and
\[
\delta_1(H_1) \ge \delta_1(H) - \frac{\alpha n}{4} (n-1) \ge \alpha \binom{n}{2} - \frac{\alpha}{2} \binom{n}{2} \ge \frac{\alpha}{2} \binom{n_1}{2} .
\]
By Lemma \ref{lem:Horev}, there are at least $c n_1^4$ $v$-absorbers in $H_1$.
By Lemma \ref{lem:shave} \ref{item:1}, $|\partial H'|\ge (1-\rho^{1/5})\binom n2$.
The desired $v$-absorber exists because
\[
c n_1^4 - 2 \rho^{1/5}\binom n2n^2\ge  c(n - \alpha n/4)^4 - \rho^{1/5}n^3(n-1) > 0. \qedhere
\]
\end{proof}

We use the following connection lemma from~\cite{Horev}.

\begin{lemma} \cite[Lemma 3.23]{Horev} \label{lem:connecting}
For every $d, \beta \in (0,1]$ with $\beta < d$, there exist an integer $n_0 >0$ and a real $\rho_0>0$ such that the following holds for all $n \ge n_0$ and $0<\rho<\rho_0$. Let $H$ be an $n$-vertex $(\rho,d)_{\cherry}$-dense $3$-graph and let $H'$ be a spanning subgraph of $H$ such that for any $x, y$ of $V(H)$, either $\deg_{H'}(xy)=0$ or $\deg_{H'}(xy) \ge \beta n$. Let $(x, y)$ and $(x', y')$ be two disjoint ordered pairs of vertices such that both $x y$ and $x' y'$ are in $\partial H'$. Then, there exists a $10$-vertex path in $H$ connecting $(x, y)$ and $(x', y')$.
\end{lemma}

%

Now we are ready to prove our absorption lemma.

\begin{lemma} \label{lem:abspath}
For every $\alpha, d\in (0,1]$, there exists $\rho>0$ such that the following holds for any sufficiently large integer $n$.
Let $H$ be an $n$-vertex $(\rho,d)_{\cherry}$-dense $3$-graph satisfying $\delta_1(H)\ge \alpha \binom{n}{2}$ and let $H'$ be a spanning subgraph of $H$ satisfying Lemma \ref{lem:shave} \ref{item:1}.
Then for any $A \subseteq V(H)$ with $|A| \le dn/66$, there exists {a path} $P$ of length at most $10|A|$ such that both ends of $P$ are in $\partial H'$, and for any $A' \subseteq A$, there is a tight path $P'$ on $V(P) \cup A'$ which has the same ends as $P$.
\end{lemma}

\begin{proof}
Apply Lemma \ref{lem:absorber} with $d$ and $\alpha$ and obtain $\rho_1$. Apply Lemma \ref{lem:connecting} with $\beta = d/6$ and obtain $\rho_2$. Let $\rho = \min \{\rho_1, \rho_2/2, d^5/(3^{40} 2^{5})\}$.
We first choose disjoint absorbers for each $v\in A$.
Let $W$ be the union of $A$ and the absorbers that have been chosen so far.
Then, for each vertex in $A$, we iteratively use Lemma \ref{lem:absorber} to find a $v$-absorber $v_1v_2v_3v_4$ in $V(H)\setminus W$ such that $\deg_{H'}(v_1v_2) \ge dn/3$ and $\deg_{H'}(v_3v_4) \ge dn/3$. 

Next, we iteratively connect these absorbers by Lemma \ref{lem:connecting} to a single tight path. At each intermediate step, let $Q$ be the union of $A$ and all the paths that have been chosen so far and suppose we need to connect $(v_3, v_4)$ and $(v'_2, v'_1)$. Define $H_1 := H[ (V(H)\setminus Q) \cup \{ v_3, v_4, v'_1, v'_2\}]$ and $H'_1 := H'[ (V(H)\setminus Q) \cup \{ v_3, v_4, v'_1, v'_2\}]$.
As all the connections will be done by Lemma \ref{lem:connecting}, we have $|V(H_1)| \ge n- |Q|\ge n-11|A|\ge 5n/6$, which implies that $H_1$ is $(2 \rho , d)_{\cherry}$-dense. 
For those $x, y$ with $\deg_{H'}(xy) =0 $, we have $\deg_{H'_1}(xy) = 0$;
for those $x, y$ with $\deg_{H'}(xy) \ge d n/3$, we have 
\[
\deg_{H'_1}(xy) \ge \deg_{H'}(xy) - 11 |A|  \ge dn/6.
\]
That is, for any $x, y$ of $V(H_1)$, either $\deg_{H'_1}(xy) = 0$ or $\deg_{H'_1}(xy) \ge dn/6$. In particular, $\deg_{H'_1}(v_3v_4) \ge d n /6$ and $\deg_{H'_1}(v'_2v'_1) \ge d n /6$. 
By Lemma \ref{lem:connecting}, there exists a $10$-vertex path in $H_1$ connecting $(v_3, v_4)$ and $(v'_2, v'_1)$. 
In conclusion, we obtain {a path} $P$ of length no more than $10|A|$.
For any $A' \subseteq A$, we can put each vertex of $A'$ into its absorber, which is an interior path of $P$. This results a tight path $P'$ on $V(P) \cup A'$ which has the same ends as $P$.
\end{proof}

At last, we use a path cover lemma given in~\cite{Horev}.
In fact the $\points$-denseness is sufficient but we state it in terms of ${\cherry}$-dense to unify the statement of the lemmas.

\begin{lemma} \cite[Lemma 1.13]{Horev} \label{lem:pathcover}
For every $d,\zeta \in (0,1]$,  there exist $n_0 := n_0(d,\zeta)$,  $\rho_0 = \rho_0(d,\zeta) >0$, and an integer $l_0= l_0(d,\zeta)$ such that the following holds for all $n \geq n_0$ and $0<\rho < \rho_0$. 
Let $H$ be an $n$-vertex $(\rho,d)_{\cherry}$-dense $3$-graph. 
Then, all but at most $ \zeta n$ vertices of $H$ can be covered using at most $l_0$ vertex-disjoint paths. 
\end{lemma}


\section{Proof of Theorem~\ref{thm:main}}

Here is a brief sketch of the proof.
We first choose a random set $A$, which will be used to deal with the leftover vertices from the almost path cover and connect all paths to a tight cycle.
Next, we apply Lemma~\ref{lem:abspath} to find an absorbing path $P_0$ for $A$, and apply Lemma \ref{lem:pathcover} to find a constant number of paths that leaves a set $U$ of vertices uncovered.
Using vertices in $A$, we put each vertex in $U$ into disjoint 5-vertex paths, which can be done by applying Lemma~\ref{lem:abspath} on $H[A\cup U]$.
Now we connect all these paths together into a tight cycle $C$, leaving only some vertices in $A$ outside $V(C)$.
Finally the uncovered vertices of $A$ will be absorbed by $P_0$ (as $P_0$ is an interior path of $C$) and we obtain a tight Hamilton cycle of $H$.

Now we start our proof.
Given $\alpha, d \in (0,1]$,
let $\sigma = \min \{\frac1{132}, \frac d{33}\}$ and $\zeta = \min\{ \frac{\alpha \sigma}{72}, \frac{d\sigma}{4320}\}$.
Apply Lemma \ref{lem:absorber} with $\alpha/18$ in place of $\alpha$, $d/6$ in place of $d$ and obtain $\rho_1$.
Apply Lemma \ref{lem:connecting} with $d$ and $\beta = d/20$ and obtain $\rho_2$.
Apply Lemma \ref{lem:abspath} with $\alpha, d$ and obtain $\rho_3$.
Apply Lemma \ref{lem:pathcover} with $d, \zeta$ and obtain $\rho_4$ and $l_0$.
Let $\rho = \min\{\rho_1\sigma^{10}/2^{10}, \rho_2 \sigma^3/27, \rho_3, \rho_4^5/32, d^5/(3^{40} 2^{5})\}$ and $n_0$ be sufficiently large.
Let $n\ge n_0$ and $H$ be an $n$-vertex $(\rho,d)_{\cherry}$-dense $3$-graph.
Let $H'$ be the spanning subgraph of $H$ given by Lemma \ref{lem:shave} satisfying \ref{item:1} and \ref{item:2}.

\medskip \noindent \textbf{Choose a random set $A$.} First, we pick a random set $A$ by including every vertex of $H$ independently with probability $\sigma$. Then, we have $\Ex (|A|) = \sigma n$.
Chernoff's inequality (see \cite[Corollary 2.3]{Janson}) and the above expectation yield that
\[
\Pro ( |A|> 2 \sigma n) = o(1), \ \ \Pro ( |A|  <\sigma n/2) = o(1).
\]
Moreover, by Janson's inequality (see \cite[Theorem 2.14]{JLR}), for any $v, x, y\in V(H)$, we have
\[
\Pro (\deg_H(v, A)< \deg_H(v) \sigma^2/2) \le 1/{n^3}, \text{ and}
\]
\[
\Pro (\deg_{H'}({xy, A}) < \deg_{H'}(xy) \sigma/2) \le 1/{n^3}.
\]
In summary, by the union bound, there exists a choice of $A$ such that
\begin{enumerate} [label=(\roman*)]
\item $ \sigma n/2 \le |A| \le 2\sigma n$, \label{romanitem:1}
\item $\deg_H(v, A) \ge \deg_H(v) \sigma^2 /2 \ge \frac{\alpha \sigma^2}{2}\binom n2$ for every $v\in V(H)$, and \label{romanitem:2}
\item either $\deg_{H'}({xy, A}) \ge \deg_{H'}(xy) \sigma/2\ge {\sigma d n}/6$ or $\deg_{H'}({xy}) =0$ for every $x, y \in V(H)$. \label{romanitem:3}
\end{enumerate}

\medskip \noindent \textbf{Pick {an absorbing path} and an almost path cover.} 
By Lemma \ref{lem:abspath}, there exists {a path} $P_0$ of length no more than $10|A|$ such that both ends $(a_0, b_0)$ and $(c_0, d_0)$ of $P_0$ are in $\partial H'$, and for any $A' \subseteq A$, there is a tight path $P'_0$ on $V(P_0) \cup A'$ which has the same ends as $P_0$.

Define $H'' := H'[ V(H)\setminus (V(P_0) \cup A)]$.
Since $|V(H'')| \ge n- |V(P_0) \cup A|\ge n-11|A|\ge 5n/6$, $H''$ is $(2 \rho^{1/5} , d)_{\cherry}$-dense.
By Lemma \ref{lem:pathcover}, all but at most $ \zeta n$ vertices of $H''$ can be covered using $l \le l_0$ vertex-disjoint paths $P_1, P_2, \dots, P_l$. 
Let $U$ be the set of uncovered vertices of $H''$.
Let $(a_i, b_i)$ and $(c_i, d_i)$ be the ends of $P_i$ for $i \in [l]$.

\medskip 
\noindent \textbf{Put vertices of $U$ into short paths.} Define $A^* := A\cup U\cup \{a_i, b_i, c_i, d_i\}_{0\le i\le l}$.
Then, we have
\[
|A^*| = |A| +|U| +| \{a_i, b_i, c_i, d_i\}_{0\le i\le l}| \le 2\sigma n +\zeta n +4(l+1) \le 3\sigma n.
\]
Since the induced subgraph $H[A^*]$ has at least $|A|\ge \sigma n/2$ vertices, $H[A^*]$ is $(8\rho/\sigma^3, d)_{\cherry}$-dense.
By \ref{romanitem:2}, for every $v \in A^*$, we have 
\[
\deg_{H[A^*]}(v) \ge \deg_{H}(v, A) \ge \frac{\alpha \sigma^2}{2}\binom n2 \ge \frac{\alpha}{18}\binom {|A^*|}{2}.
\]
Let $H'[A^*]$ be the spanning subgraph of $H[A^*]$. For $x, y \in A^*$, by \ref{romanitem:3}, we have either
\[
\deg_{H'[A^*]}(xy) \ge \deg_{H'}(xy, A) \ge {\sigma d n}/6 \ge \frac{d}{18} {|A^*|},
\]
or $\deg_{H'}({xy}) =0$.
Moreover, the number of $xy$ with $\deg_{H'}({xy}) =0$ is at most
\[
\rho^{1/5} \binom n2 \le \rho^{1/5} \binom {2|A^*|/\sigma}{2} \le \frac{4\rho^{1/5} }{\sigma^2}\binom {|A^*|}{2},
\]
as $\sigma n/2 \le |A^*|$.
We apply Lemma \ref{lem:absorber} on $H[A^*]$ with $\alpha/18$ in place of $\alpha$, $d/6$ in place of $d$, and $H'[A^*]$ playing the role of $H'$, and conclude that for any $W \subseteq A^*$ with $|W| \le \alpha |A^*| /72$ and any $v \in U$, there exists a $v$-absorber $v_1v_2v_3v_4$ in $A^*\setminus W$ such that $\deg_{H'[A^*]}(v_1v_2) \ge d |A^*|/18$ and $\deg_{H'[A^*]}(v_3v_4) \ge d |A^*|/18$.

We greedily choose disjoint paths $Q_v = v_1v_2vv_3v_4$ for each $v\in U$ (clearly, a $v$-absorber gives such a path) such that $\deg_{H'[A^*]}(v_1v_2) \ge d |A^*|/18$ and $\deg_{H'[A^*]}(v_3v_4) \ge d |A^*|/18$.
Let $W$ be the union of $U \cup \{a_i, b_i, c_i, d_i\}_{0\le i\le l}$ and the paths that have been chosen so far.
Note that $|W| \le |U|+4(l+1)\le \zeta n +4l +4\le \alpha |A^*| /72$.
Then, for every vertex $v\in U$, we can use the property above to find a $v$-absorber $v_1v_2v_3v_4$ in $A^*\setminus W = A\setminus W$ such that $\deg_{H'[A^*]}(v_1v_2) \ge d|A^*|/18$ and $\deg_{H'[A^*]}(v_3v_4) \ge d|A^*|/18$. 

\medskip 
\noindent \textbf{Connect the paths and finish the absorption.} 
Next, we iteratively connect these paths $P_0, P_1, \dots,P_l$, and $\{Q_v: v \in U\}$  to a tight cycle. 
At each intermediate step, let $Q$ be the vertex set of the union of all the paths that have been chosen so far and suppose we need to connect two ends $(z_1, z_2)$ and $(w_1, w_2)$. 
Because we will connect these paths by Lemma \ref{lem:connecting}, we have 
\[
|A^* \cap Q|\le 5\zeta n +4l +4+6(\zeta n +l+1)\le 12\zeta n \le 24(\zeta/\sigma)|A^*| \le |A^*|/10.
\]
Define $H_1 := H[ (A^*\setminus Q) \cup \{ z_1, z_2, w_1, w_2\}]$ and $H'_1 := H'[ (A^*\setminus Q) \cup \{ z_1, z_2, w_1, w_2\}]$.
Since 
\[|V(H_1)| \ge |A^*|- |A^* \cap Q| \ge (9/10)|A^*| \ge 9\sigma n/20 > \sigma n/3,
\]
$H_1$ is $(27\rho/\sigma^3 , d)_{\cherry}$-dense. 
Recall that for $x, y \in A^*$, either $\deg_{H'[A^*]}(xy)\ge \frac{d}{18} {|A^*|}$ or $\deg_{H'}({xy}) =0$.
For those $x, y$ with $\deg_{H'}(xy) =0 $, we have $\deg_{H'_1}(xy) = 0$; otherwise,
\[
\deg_{H'_1}(xy) \ge \deg_{H'[A^*]}(xy) - |A^* \cap Q|  \ge \frac{d}{18} {|A^*|} - 24(\zeta/\sigma)|A^*| \ge \frac{d}{20} {|A^*|}.
\]
That is, for any $x, y$ of $A^*$, either $\deg_{H'_1}(xy) = 0$ or $\deg_{H'_1}(xy) \ge d|A^*|/20$. 
In particular, $\deg_{H'_1}(z_1z_2) \ge d|A^*|/20$ and $\deg_{H'_1}(w_1w_2) \ge d|A^*|/20$. 
By Lemma \ref{lem:connecting}, there exists a $10$-vertex path in $H_1$ connecting $(z_1, z_2)$ and $(w_1, w_2)$. 
In conclusion, we obtain a tight cycle that covers all vertices in $V(H)\setminus A$.
As the uncovered vertices are all in $A$ and can be absorbed by $P_0$, we obtain a Hamilton cycle and the proof is completed.

\bibliographystyle{abbrv}
\bibliography{refs,Bibref}

\end{document}